\documentclass[11pt,leqno]{amsart}
\usepackage{amsmath,amssymb,amsthm}
\usepackage{hyperref}

\usepackage{bbm}

\DeclareMathOperator{\conv}{conv}

\renewcommand{\geq}{\geqslant}
\renewcommand{\leq}{\leqslant}

\DeclareMathOperator{\real}{Re}

\newcommand{\Lip}{{\mathrm{Lip}}_0}
\newcommand{\justLip}{{\mathrm{Lip}}}

\newcommand{\spann}{\operatorname{span}}

\newcommand{\1}[1]{\operatorname{\textbf{1}}}

\newtheorem{theorem}{Theorem}[section]
\newtheorem{lemma}[theorem]{Lemma}

\newtheorem{proposition}[theorem]{Proposition}
\newtheorem{corollary}[theorem]{Corollary}
\theoremstyle{definition}
\newtheorem{definition}[theorem]{Definition}
\newtheorem{example}[theorem]{Example}

\theoremstyle{remark}
\newtheorem{remark}[theorem]{Remark}
\numberwithin{equation}{section}

\def\fnote#1{\footnote}

\def\ignora#1{}
\def\n3#1{\left\vert  \! \left\vert \! \left\vert \, #1 \, \right\vert \!
  \right\vert \! \right\vert }

\begin{document}

\title{ Slice diameter two property in ultrapowers }

\author{ Abraham Rueda Zoca }\address{Universidad de Granada, Facultad de Ciencias. Departamento de An\'{a}lisis Matem\'{a}tico, 18071-Granada
(Spain)} \email{ abrahamrueda@ugr.es}
\urladdr{\url{https://arzenglish.wordpress.com}}



\maketitle

\markboth{ABRAHAM RUEDA ZOCA}{SLICE DIAMETER TWO PROPERTY IN ULTRAPOWERS}

\begin{abstract}
In this note we study the inheritance of the slice diameter two property by ultrapower spaces. Given a Banach space $X$, we give a characterisation of when $(X)_\mathcal U$, the ultrapower of $X$ through a free ultrafilter $\mathcal U$, has the slice diameter two property obtaining that this is the case for many Banach spaces which are known to enjoy the slice diameter two property. We also provide, for every $\eta>0$, an example of a Banach space $X$ with the Daugavet property such that the unit ball of $(X)_\mathcal U$ contains a slice of diameter smaller than $\eta$ for every free ultrafilter $\mathcal U$ over $\mathbb N$. This proves, in particular, that the slice diameter two property is not in general inherited by taking ultrapower spaces.
\end{abstract}

\section{Introduction}

Ultrapowers of Banach spaces have been intensively studied in the literature as they have proved to be a useful tool in order to study local theory of Banach spaces (as a matter of fact, ultraproducts are used in \cite[Chapter 11]{alka2006} in order to prove that $\ell_1$ is finitely representable in $X$ if, and only if, $X$ fails to have type $p>1$).

In addition to this, different topological and geometrical properties of Banach spaces have been studied in ultrapowers of Banach spaces. Concerning the topological ones, we have for instance the study of reflexivity in ultrapowers (see e.g. \cite{james}) or the weak compactness of sets in ultrapowers \cite{greraj23,Tu}. On the other hand, different geometrical properies of Banach spaces have been analysed in ultrapower spaces (like the property of being (isometrically) an $L_1$-predual space \cite{hein81}, the study of extreme points of the unit ball \cite{ggr22,talponen17}, the study of strongly exposed points of the unit ball \cite{ggr22} or the property of being \textit{almost square} Banach space (see \cite{Hardtke18} for definition and details)).

A classical result about ultrapower spaces is the following: given a Banach space $X$ and a free ultrafilter $\mathcal U$ over $\mathbb N$, it follows that $(X_\mathcal U)^*=(X^*)_\mathcal U$ if, and only if, $X$ is superreflexive. Moreover, if $X$ is not superreflexive, there is not a good description of the topological dual of $X_\mathcal U$. Because of this reason, informally speaking, properties of Banach spaces which are described using elements of the topological dual may be difficult to analyse in ultrapower spaces. This is the case, for instance, for properties which deal with the behaviour of the slices of the unit ball (see Section 2 for details), like the \textit{slice diameter two property}.

A Banach space $X$ is said to have the \textit{slice diameter two property (slice-D2P)} if every slice of $B_X$ has diameter exactly two. We refer the interested reader to \cite{ahltt16,blr15,rueda23} and references therein for background on the topic. 

This property has been widely studied during the last 25 years but, as far as the author knows, little is known about when the slice-D2P passes on from a Banach space to its ultrapowers. Let us point out that, from the study of stronger properties of Banach spaces, some ultrapower spaces are known to enjoy the slice-D2P. For instance, in \cite{Hardtke18} it is shown that $(X)_\mathcal U$ has the slice-D2P whenever the space $X$ is \textit{locally almost square}, a property which is strictly stronger than the slice-D2P (see Example \ref{examp:hardtke} for details). Moreover, examples of ultrapowers with the slice-D2P come from ultrapowers actually satisfying the \textit{Daugavet property}.

Let us formally introduce the Daugavet property. We say that a Banach space $X$ has the \textit{Daugavet property} if, for every slice $S$ of $B_X$, every $x\in S_X$ and every $\varepsilon>0$, there exists $y\in S$ satisfying
$$\Vert x-y\Vert>2-\varepsilon.$$
We refer the reader to \cite{kkw03,kssw01,shvidkoy00, werner01} and references therein for background. It is clear from the definition that Banach spaces with the Daugavet property enjoy the slice-D2P.

We pay attention to the Daugavet property as a paradigm in the following sense: given a Banach space $X$, the study of the Daugavet property implies to deal with slices of the unit ball (and consequently with elements of $X^*$) so, at a first glance, one could expect a big difficulty in the analysis of the Daugavet property in an ultrapower space. However, a complete characterisation of when an ultrapower space has the Daugavet property was obtained in \cite{bksw}.

The key idea was to make use of a characterisation of the Daugavet property which avoids the use of slices: a Hahn-Banach separation argument implies that $X$ has the Daugavet property if, and only if, $B_X=\overline{\conv}\{y\in B_X: \Vert x-y\Vert>2-\varepsilon\}$ holds for every $x\in S_X$ and every $\varepsilon>0$ (c.f. e.g. \cite[Lemma 2.3]{werner01}).

With this idea in mind, the authors of \cite{bksw} considered a uniform version of the Daugavet property, the so called \textit{uniform Daugavet property} (see \cite[p. 59]{bksw}), and they characterised those Banach spaces $X$ for which $(X)_\mathcal U$ has the Daugavet property. They also showed that all the classical examples of Banach spaces with the Daugavet property actually satisfy its uniform version. In \cite{kw04}, however, the authors constructed a Banach space $X$ with the Daugavet and the Schur properties such that $(X)_\mathcal U$ fails the Daugavet property for every free ultrafilter $\mathcal U$ over $\mathbb N$.

In this note our starting point will be a characterisation of the slice-D2P in the spirit of the above mentioned \cite[Lemma 2.3]{werner01} coming from \cite{ivakhno06}: a Banach space $X$ has the slice-D2P if, and only if, $B_X=\overline{\conv}\{\frac{x+y}{2}: x,y\in B_X, \Vert x-y\Vert>2-\varepsilon\}$ holds for every $\varepsilon>0$.

Using the above, in Theorems \ref{theo:necebigdiamultrapower} and \ref{theo:sufsliced2pultrapower} we completely characterise when, given a sequence $(X_n)_{n\in\mathbb N}$ of Banach spaces and a free ultrafilter $\mathcal U$ over $\mathbb N$, the ultraproduct $(X_n)_\mathcal U$ has the slice-D2P in terms of requiring that all $X_n$ have the slice-D2P ``in a uniform way''. This motivates us to introduce the \textit{uniform slice diameter two property} in Definition \ref{defi:uniformsliced2p}, showing that this property is enjoyed by most of the classical spaces which are known to have the slice-D2P. All this is discussed in Section \ref{section:uniformsliced2p}. 

In Section \ref{section:counterexample} we will have a look to the involved construction from \cite{kw04} of a Daugavet space whose ultrapowers fail the Daugavet property. We will make use of the above example in order to construct, for every $\eta>0$, a Banach space with the Daugavet property such that the unit ball of $(X)_\mathcal U$ contains slices of diameter smaller than $\eta$ for every free ultrafilter $\mathcal U$ over $\mathbb N$. This will show, in particular, that the slice-D2P is not in general inherited by taking ultrapower spaces.

\section{Notation and preliminary results}

We will consider Banach spaces over the scalar field $\mathbb R$ or $\mathbb C$. 

Given a Banach space $X$ then $B_X$ (respectively $S_X$) stands for the closed unit ball (respectively the unit sphere) of $X$. We will denote by $X^*$ the topological dual of $X$. Given a subset $C$ of $X$, we will denote by $\conv(C)$ the convex hull of $C$ and by $\spann(C)$ the linear span of $C$. We also denote, given $n\in\mathbb N$, the set 
$$\conv_n(C):=\left\{\sum_{i=1}^n \lambda_i x_i: \lambda_1,\ldots, \lambda_n\in [0,1], \sum_{i=1}^n \lambda_i=1, x_1,\ldots, x_n\in C \right\}.$$
In other words, $\conv_n(C)$ stands for the set of all convex combinations of at most $n$ elements of $C$.

If $C$ is a bounded set, by a \textit{slice}\index{slice} of $C$ we will mean a set of the following form
$$S(C,f,\alpha):=\{x\in C: \real f(x)>\sup \real f(C)-\alpha\}$$
where $f\in X^*$ and $\alpha>0$. Notice that a slice is nothing but the intersection of a half-space with the bounded (and not necessarily convex) set $C$.

In \cite[Lemma 1]{ivakhno06} it is proved that a Banach space $X$ has the slice-D2P if, and only if, $B_X:=\overline{\conv}\{\frac{x+y}{2}: \Vert x-y\Vert>2-\varepsilon\}$ holds for every $\varepsilon>0$. Indeed, we state here for future reference the following more general version, which was already observed in \cite[Section 5]{lmr24}. Since the above mentioned \cite{lmr24} deals only with real Banach spaces, we include a complete proof of the following proposition to cover the complex case too and for the sake of completeness.

\begin{proposition}\label{prop:charaslideDdeltap}
Let $X$ be a Banach space. The following are equivalent:
\begin{enumerate}
    \item Every slice of $B_X$ has diameter, at least, $\alpha$.
    \item $B_X=\overline{\conv}\left\{\frac{x+y}{2}: x,y\in B_X, \Vert x-y\Vert\geq \alpha-\varepsilon\right\}$ holds for every $\varepsilon>0$.
\end{enumerate}
\end{proposition}

\begin{proof}
(1)$\Rightarrow$(2). Assume that (2) does not hold. Then there exists $\varepsilon>0$ and $x_0\in B_X$ such that $x_0\notin \overline{\conv}\left\{\frac{x+y}{2}: x,y\in B_X, \Vert x-y\Vert\geq \alpha-\varepsilon\right\}$. Call $A:=\left\{\frac{x+y}{2}: x,y\in B_X, \Vert x-y\Vert\geq \alpha-\varepsilon\right\}$. By Hahn-Banach theorem we can find a slice $S$ of $B_X$ such that $x_0\in S$ and $S\cap A=\emptyset$. We claim that if $u,v\in S$ it follows that $\Vert u-v\Vert<\alpha-\varepsilon$. Indeed, if there existed $u,v\in S$ with $\Vert u-v\Vert\geq \alpha-\varepsilon$, then $\frac{u+v}{2}$ would belong to $S$ by the convexity of $S$. Since clearly $\frac{u+v}{2}\in A$ we would get that $S\cap A\neq \emptyset$, which is impossible. This proves that $\Vert u-v\Vert\leq \alpha-\varepsilon$ holds for every $u,v\in S$, which proves the negation of (1).

(2)$\Rightarrow$(1). Take a slice $S:=S(B_X,x^*,\beta)$, where $x^*\in S_{X^*}$ and $\beta>0$, and let $\varepsilon>0$, and let us prove that there are $u,v\in S$ such that $\Vert u-v\Vert\geq \alpha-\varepsilon$. The arbitrariness of $\varepsilon$ will imply (1). In order to do so, consider the slice $S(B_X,x^*,\frac{\beta}{2})$. Since $\overline{\conv}\left\{\frac{x+y}{2}: x,y\in B_X, \Vert x-y\Vert\geq \alpha-\varepsilon\right\}=B_X$ we infer that $S(B_X,x^*,\frac{\beta}{2})\cap \left\{\frac{x+y}{2}: x,y\in B_X, \Vert x-y\Vert\geq \alpha-\varepsilon\right\}\neq \emptyset$ (since the complement in $B_X$ of slices are clearly convex sets). Consequently, we can find $u,v\in B_X$ with $\Vert u-v\Vert\geq \alpha-\varepsilon$ and such that $\frac{u+v}{2}\in S(B_X,x^*,\frac{\beta}{2}$. In order to finish the proof, let us prove that both $u,v\in S=S(B_X,x^*,\beta)$ which means, by definition, that $\real x^*(u)>1-\beta$ and $\real x^*(v)>1-\beta$. To this end observe that, $\frac{u+v}{2}\in S(B_X,x^*,\frac{\beta}{2})$ means $\real x^*\left(\frac{u+v}{2}\right)>1-\frac{\beta}{2}$. Now
$$1-\frac{\beta}{2}\leq \frac{\real x^*(u)+\real x^*(v)}{2}\leq \frac{\real x^*(u)+\Vert x^*\Vert}{2}=\frac{\real x^*(u)+1}{2}.$$
This implies $\real x^*(u)+1\geq 2-\beta$, from where $\real x^*(u)>1-\beta$. In a similar way, it is proved that $\real x^*(v)>1-\beta$, which means $u,v\in S$, as desired.
\end{proof}

The above result motivates us to introduce the following notation, which will be useful throughout the text. Given a Banach space $X$ and $\alpha>0$, define
$$S^\alpha(X):=\left\{\frac{x+y}{2}: x,y\in B_X, \Vert x-y\Vert\geq  \alpha\right\}.$$
Given $n\in\mathbb N$ we denote
$$S_n^\alpha(X):=\conv_n(S^\alpha(X)).$$
Finally, given $n\in\mathbb N$ and $\alpha>0$, we define
$$C_n^\alpha(X):=\sup_{x\in S_X} d(x,S_n^\alpha(X))=\sup_{x\in S_X}\inf_{y\in S_n^\alpha(X)} \Vert x-y\Vert.$$
It follows from the very definition of $C_n^\alpha(X)$ the following two properties:
\begin{enumerate}
    \item Given $0<\alpha<\beta$ then $C_n^\alpha\geq C_n^\beta$ and,
    \item given two natural numbers $n\geq m$ then $C_n^\alpha(X)\leq C_m^\alpha(X)$.
\end{enumerate}

Given a sequence of Banach spaces $\{X_n:n\in\mathbb N\}$ we denote 
$$\ell_\infty(\mathbb N,X_n):=\left\{f\colon \mathbb N\longrightarrow \prod\limits_{n\in \mathbb N} X_n: f(n)\in X_n\ \forall n\text{ and }\sup_{n\in \mathbb N}\Vert f(n)\Vert<\infty\right\}.$$
Given a free ultrafilter $\mathcal U$ over $\mathbb N$, consider $c_{0,\mathcal U}(\mathbb N,X_n):=\{f\in \ell_\infty(\mathbb N,X_n): \lim_\mathcal U \Vert f(n)\Vert=0\}$. The \textit{ultraproduct of $\{X_n:n\in\mathbb N\}$ with respect to $\mathcal U$} is
the Banach space
$$(X_n)_\mathcal U:=\ell_\infty(\mathbb N,X_n)/c_{0,\mathcal U}(\mathbb N,X_n).$$
We will naturally identify a bounded function $f\colon\mathbb N\longrightarrow \prod\limits_{n\in \mathbb N} X_n$ with the element $(f(n))_{n\in\mathbb N}$. In this way, we denote by $(x_n)_\mathcal U$ or simply by $(x_n)$, if no confusion is possible, the coset in $(X_n)_\mathcal U$ given by $(x_n)_{n\in \mathbb N}+c_{0,\mathcal U}(\mathbb N,(X_n))$.

From the definition of the quotient norm, it is not difficult to prove that $\Vert (x_n)\Vert=\lim_\mathcal U \Vert x_n\Vert$ holds for every $(x_n)\in (X_n)_\mathcal U$.

\section{Uniform slice-D2P}\label{section:uniformsliced2p}

Let us start by looking for necessary conditions for an ultraproduct space to enjoy the slice-D2P. In order to do so, as announced before, we will make use of Proposition \ref{prop:charaslideDdeltap}.

\begin{theorem}\label{theo:necebigdiamultrapower}
    Let $(X_n)$ be a sequence of Banach spaces, $\mathcal U$ be a free ultrafilter over $\mathbb N$ and $\alpha>0$. Set $X:=(X_n)_\mathcal U$ and assume that every slice of $B_X$ has diameter at least $\alpha$. Then, for every $\delta> 0$ and $\varepsilon>0$ there exists $n\in\mathbb N$ such that
$$\{k\in\mathbb N: C_n^{\alpha-\varepsilon}(X_k)<\delta\}\in\mathcal U.$$

\end{theorem}

\begin{proof}
Assume that there exist $\delta_0>0, \varepsilon_0>0$ such that, for every $n\in\mathbb N$, we get 
$$\{k\in\mathbb N: C_n^{\alpha-\varepsilon_0}(X_k)\geq \delta_0\}\in\mathcal U.$$
We can select, for every $n\geq 2$, a set $A_n\subseteq \{k\in\mathbb N: C_n^{\alpha-\varepsilon_0}(X_k)\geq \delta_0\}$ such that $A_n\in\mathcal U$ holds for every $n\in\mathbb N$, $\bigcap\limits_{n\geq 2}A_n=\emptyset$ and $A_{n+1}\subseteq A_n$ holds for $n\geq 2$. Take $A_1=\mathbb N$. Observe that $\{A_n\setminus A_{n+1}: n\in\mathbb N\}$ is a partition of $\mathbb N$. Moreover, for every $n\geq 2$, for every $p\in A_n\setminus A_{n+1}$ we can find $x_p\in S_{X_p}$ satisfying that $d(x_p,S_n^{\alpha-\varepsilon_0}(X_p))\geq \frac{\delta_0}{2}$. For $p\in A_1\setminus A_2$ select any $x_p\in S_{X_p}$.

Now $(x_p)\in S_{X}$. We claim that $d((x_p), \conv (S^{\alpha-\frac{\varepsilon_0}{2}}(X)))\geq \frac{\delta_0}{2}$. Once this is proved, Proposition \ref{prop:charaslideDdeltap} implies that there exists a slice in $(X_n)_\mathcal U$ of diameter smaller than $\alpha$, which will finish the proof of the theorem. In order to do so, take $z\in \conv(S^{\alpha-\frac{\varepsilon_0}{2}}(X))$, so there is $q\in\mathbb N$ such that $z\in \conv_q(S^{\alpha-\frac{\varepsilon_0}{2}}(X))$. 

By definition we can find $\lambda_1,\ldots, \lambda_q\in [0,1]$ with $\sum_{i=1}^q\lambda_i=1$ and $(u_n^i),(v_n^i)\in S_X$ with $\Vert (u_n^i)-(v_n^i)\Vert\geq \alpha-\frac{\varepsilon_0}{2}$ and $z=\sum_{i=1}^q \lambda_i \frac{(u_n^i)+(v_n^i)}{2}$. Let $\eta>0$. Since $\Vert (x_n)-(z_n)\Vert=\lim_\mathcal U \Vert x_n-z_n\Vert$ , the set
$$B:=\{n\in\mathbb N: \vert \Vert x_n-z_n\Vert-\Vert (x_n)-(z_n)\Vert \vert<\eta\}\in\mathcal U.$$
On the other hand, given $1\leq i\leq q$ it follows that $\lim_\mathcal U \Vert u_n^i-v_n^i\Vert\geq \alpha-\frac{\varepsilon_0}{2}>\alpha-\varepsilon_0$. This implies that the set
$$C:=\bigcap\limits_{i=1}^q\left\{n\in\mathbb N: \Vert u_n^i-v_n^i\Vert>\alpha-\varepsilon_0\right\}\in\mathcal U.$$
Select any $k\in A_q\cap B\cap C$. Then, since $k\in B$, we have
$$\Vert (x_n)-(z_n)\Vert\geq \Vert x_k-z_k\Vert-\eta.$$
On the other hand, $z_k=\sum_{i=1}^q \lambda_i \frac{u_k^i+v_k^i}{2}$ with $\Vert u_k^i-v_k^i\Vert\geq\alpha-\varepsilon_0$ since $k\in C$. Hence, $z_k\in \conv_q(S^{\alpha-\varepsilon_0}(X_k))$. Finally, since $k\in A_q$ we conclude that $\Vert x_k-z_k\Vert\geq \frac{\delta_0}{2}$, so
$$\Vert (x_n)-(z_n)\Vert\geq \frac{\delta_0}{2}-\eta.$$
The arbitrariness of $\eta>0$ and $(z_k)\in \conv(S^{\alpha-\frac{\varepsilon_0}{2}}(X))$ implies that \\ $d((x_n),\conv(S^{\alpha-\frac{\varepsilon_0}{2}}(X)))\geq \frac{\delta_0}{2}$, as desired.
\end{proof}
In the following result we establish the converse.

\begin{theorem}\label{theo:sufsliced2pultrapower}
Let $(X_n)$ be a sequence of Banach spaces, $0<\alpha<2$ and a free ultrafilter $\mathcal U$ over $\mathbb N$. Assume that for every $\delta>0$ there exists $n\in\mathbb N$ such that 
$$\{k\in\mathbb N: C_n^\alpha(X_k)<\delta\}\
\in\mathcal U.$$
Then, every slice of $(X_n)_\mathcal U$ contains two points at distance at least $\alpha$.
\end{theorem}

\begin{proof}
Let $(x_n)\in S_{(X_n)_\mathcal U}$ and let us prove, in virtue of Proposition~\ref{prop:charaslideDdeltap}, that 
$$(x_n)\in \overline{\conv}\left(\left\{\frac{(u_n)+(v_n)}{2}: (u_n),(v_n)\in B_{(X_n)_\mathcal U}, \Vert (u_n)-(v_n)\Vert\geq \alpha \right\}\right).$$
In order to do so take $\delta>0$. By the assumption there exists $n\in\mathbb N$ such that
$$A:=\{k\in\mathbb N: C_n^\alpha(X_k)<\delta\}\in\mathcal U.$$
Consequently, for every $k\in A$ there exists $\sum_{i=1}^n \lambda_i^k \frac{u_k^i+v_k^i}{2}$ such that
$$\left\Vert x_k-\sum_{i=1}^n \lambda_i^k \frac{u_k^i+v_k^i}{2} \right\Vert<\delta$$
and 
$$\Vert u_k^i-v_k^i\Vert\geq \alpha$$
holds for every $1\leq i\leq n$.

Since $\lambda_i^k\in [0,1]$ then consider $\lambda_i:=\lim_{k,\mathcal U} \lambda_i^k\in [0,1]$. It is not difficult to prove that $\sum_{i=1}^n \lambda_i=1$. Now, given $1\leq i\leq n$ define
$$u_k^i=v_k^i=0\ \forall k\notin A.$$
It is immediate that $(u_k^i), (v_k^i)\in B_{(X_n)_\mathcal U}$. Let us start by proving that $\Vert (u_k^i)-(v_k^i)\Vert=\lim_\mathcal U \Vert  u_k^i- v_k^i\Vert\geq \alpha$ holds for $1\leq i\leq n$. In order to do so, fix $\eta>0$ and $1\leq i\leq n$. By definition of limit through $\mathcal U$ and the definition of the norm of ultraproducts the set
$$B_\eta:=\{p\in\mathbb N: \vert \Vert ( u_k^i)-( v_k^i)\Vert-\Vert  u_p^i- v_p^i\Vert \vert<\eta\}\in\mathcal U.$$
Consequently, given $p\in B_\eta\cap A$ we obtain
\[
\begin{split}
\Vert ( u_k^i)-( u_k^i)\Vert\mathop{\geq}\limits^{\tiny{p\in B_\eta}} \Vert  u_p^i- v_p^i\Vert-\eta\mathop{\geq}\limits^{\tiny{p\in A}}\alpha-\eta.
\end{split}
\]
The arbitrariness of $\eta>0$ implies $\Vert ( u_k^i)-( v_k^i)\Vert\geq \alpha$.

Now it is time to prove that
$$\left\Vert (x_k)-\sum_{i=1}^n \lambda_i \frac{( u_k^i)+( v_k^i)}{2}\right\Vert\leq \delta.$$
In order to do so, take $\nu>0$. Set 
$$C_\nu:=\left\{p\in\mathbb N: \left \vert \left\Vert (x_k)-\sum_{i=1}^n \lambda_i \frac{( u_k^i)+( v_k^i)}{2}\right\Vert-\left\Vert x_p -\sum_{i=1}^n \lambda_i \frac{ u_p^i+v_p^i}{2}\right\Vert \right\vert<\nu \right\}\in\mathcal U.$$
On the other hand set
$$D:=\bigcap\limits_{i=1}^n\left\{p\in\mathbb N: \left\vert \lambda_i^p-\lambda_i \right\vert<\frac{\nu}{n}\right\}\in\mathcal U. $$
Now given $p\in C_\nu\cap D\cap A$ we get
\[
\begin{split}
\left\Vert (x_k)-\sum_{i=1}^n \lambda_i\frac{( u_k^i)+( v_k^i)}{2}\right\Vert& \mathop{\leq}\limits^{\tiny{p\in C_\nu}}\nu +\left\Vert x_p-\sum_{i=1}^n \lambda_i \frac{u_p^i+v_p^i}{2} \right\Vert\\
& \leq \nu +\left\Vert x_p-\sum_{i=1}^n \lambda_i^p \frac{u_p^i+v_p^i}{2} \right\Vert+\sum_{i=1}^n \vert \lambda_i-\lambda_i^p\vert\\
& \mathop{\leq}\limits^{\tiny{p\in A}}\nu+\delta+\sum_{i=1}^n \vert \lambda_i^p-\lambda_i\vert\\
& \mathop{\leq}\limits^{\tiny{p\in D}}2\nu+\delta.
\end{split}
\]
The arbitrariness of $\nu>0$ proves $\left\Vert (x_n)-\sum_{i=1}^n \lambda_i \frac{( u_k^i)+( v_k^i)}{2}\right\Vert\leq \delta$, which finishes the proof.
\end{proof}

Going back to the slice diameter two property, given a Banach space $X$ we have that, a combination of Theorems \ref{theo:necebigdiamultrapower} and \ref{theo:sufsliced2pultrapower} together with the fact that $(C_n^\alpha(X))_{n\in\mathbb N}$ is a decreasing sequence, yield the following corollary.

\begin{corollary}\label{coro:carasliced2pultrap}
Let $X$ be a Banach space. The following are equivalent:
\begin{enumerate}
    \item $(X)_\mathcal U$ has the slice-D2P for every free ultrafilter $\mathcal U$ over $\mathbb N$.
    \item For every $0<\alpha<2$, $\lim_{n\rightarrow \infty} C_n^\alpha(X)=0$.
\end{enumerate}
\end{corollary}

\begin{proof}
(2)$\Rightarrow$(1). Let $\mathcal U$ be a free ultrafilter over $\mathbb N$ and $0<\alpha<2$. Let us prove that every slice of the unit ball of $B_{(X)_\mathcal U}$ contains two points at distance at least $\alpha$, for which we will make use of Theorem~\ref{theo:sufsliced2pultrapower}. In order to do so, let $\delta>0$. Since $\lim_n C_n^\alpha(X)=0$ then we can find $m\in\mathbb N$ such that $C_m^\alpha(X)<\delta$. Consequently, if we take $X_n=X$, we get that
$$\{n\in\mathbb N: C_m^\alpha(X_n)<\delta\}=\mathbb N\in\mathcal U.$$
Hence we get that $\{n\in\mathbb N: C_m^\alpha(X_n)<\delta\}\in\mathcal U$. The arbitrariness of $\delta>0$ yields the conclussion.

(1)$\Rightarrow$(2). Take $\alpha>0$. In order to prove that $\lim_{n\rightarrow \infty} C_n^\alpha(X)=0$ select $\delta>0$ and let us find $m\in\mathbb N$ such that $C_n^\alpha(X)<\delta$ holds for every $n\geq m$. To do so, select any free ultrafilter $\mathcal U$ over $\mathbb N$. Since every slice of the unit ball of $(X)_\mathcal U$ has diameter at least $\alpha$, Theorem~\ref{theo:necebigdiamultrapower} implies that there exists some $m\in\mathbb N$ such that $C_m^\alpha(X)<\delta$. Since $(C_n^\alpha(X))$ is a decreasing sequence we get that $C_n^\alpha(X)\leq C_m^\alpha(X)<\delta$ holds for every $n\geq m$. The above condition together with the clear fact that $C_n^\alpha(X)\geq 0$ holds for every $n\in\mathbb N$ imply that $\lim_{n\rightarrow\infty}C_n^\alpha(X)=0$, which finishes the proof.
\end{proof}

Corollary \ref{coro:carasliced2pultrap} motivates the following definition.

\begin{definition}\label{defi:uniformsliced2p}
Let $X$ be a Banach space. We say that $X$ has the uniform slice diameter two property (uniform slice-D2P) if, for every $0<\alpha <2$, 
$$\lim_n C_n^\alpha(X)=0.$$
\end{definition}

The rest of this section will be devoted to providing examples of Banach spaces with the uniform slice-D2P.

\begin{example}\label{exam:L1spaces}
Let $X=L_1(\mu)$. It follows that $X$ has the slice-D2P if, and only if, $\mu$ contains no atom (c.f. e.g. \cite[Theorem 2.13 (ii)]{belo2004}). But if $\mu$ is an atomless measure it follows that $(X)_\mathcal U$ has the Daugavet property for every free ultrafilter $\mathcal U$ \cite[Lemma 6.6 and Theorem 6.2]{bksw}. In particular, $(X)_\mathcal U$ has the slice-D2P for every free ultrafilter $\mathcal U$.

Consequently, an $L_1$ space has the slice-D2P if, and only if, it satisfies the uniform slice-D2P.
\end{example}

More examples of spaces enjoying the uniform slice-D2P come from ultrapower spaces with the slice-D2P.

\begin{example}
Let $X$ be a Banach space with the uniform slice-D2P and let $\mathcal U$ be a free ultrafilter over $\mathbb N$. We claim that $(X)_\mathcal U$ has the uniform slice-D2P. In order to prove this it is enough to prove that, given any ultrafilter $\mathcal V$ over $\mathbb N$ then $(X_\mathcal U)_\mathcal V$ has the slice-D2P. However, this result follows since $X$ has the uniform slice-D2P and $(X_\mathcal U)_\mathcal V$ is isometrically isomorphic to $(X)_\mathcal W$ where $\mathcal W$ is a free ultrafilter over $\mathbb N$. Indeed, $\mathcal W=\mathcal U\times \mathcal V$ (see \cite[Proposition 1.2.7]{grelier} for details).
\end{example}

Another class where the slice-D2P and its uniform version are equivalent is the one of $L_1$-preduals.

\begin{example}
Let $X$ be an $L_1$ predual. Observe that $X$ has the slice-D2P if, and only if, $X$ is infinite-dimensional (c.f. e.g. \cite[Corollary 2.9]{belo2004}). Since the ultrapower of any $L_1$ predual is again an $L_1$ predual by \cite[Proposition 2.1]{hein81}, it follows that $(X)_\mathcal U$ has the slice-D2P for every free ultrafilter $\mathcal U$ as soon as $X$ has the slice-D2P, from where the uniform slice-D2P follows on $X$.
\end{example}

The following examples will come from \cite{Hardtke18}, for which we need to introduce a bit of notation. According to \cite{all16}, a Banach space $X$ is
\begin{enumerate}
\item
  \emph{locally almost square} (LASQ) if for every $x\in S_X$
  there exists a sequence $\{y_n\}$ in $B_X$ such that
  $\Vert x\pm y_n\Vert\rightarrow 1$ and $\Vert y_n\Vert\rightarrow 1$.
\item
  \emph{weakly almost square} (WASQ) if for every $x\in S_X$
  there exists a sequence $\{y_n\}$ in $B_X$ such that
  $\Vert x\pm y_n\Vert\rightarrow 1$, $\Vert y_n\Vert\rightarrow 1$
  and $y_n \rightarrow 0$ weakly.
\item
  \emph{almost square} (ASQ) if for every $x_1,\ldots, x_k \in S_X$
  there exists a sequence $\{y_n\}$ in $B_X$ such
  that $\Vert y_n\Vert\rightarrow 1$ and $\Vert x_i\pm y_n\Vert\rightarrow 1$
  for every $i\in\{1,\ldots, k\}$.
\end{enumerate}

We refer the reader to \cite{all16, gr17, roru23} and references therein for examples of LASQ, WASQ and ASQ Banach spaces. 

\begin{example}\label{examp:hardtke}
If $X$ is LASQ then $X$ has the uniform slice-D2P. Indeed, if $X$ is LASQ then $(X)_\mathcal U$ is LASQ for every free ultrafilter over $\mathbb N$ by \cite[Proposition 4.2]{Hardtke18}. The result follows since LASQ spaces have the slice-D2P (c.f. e.g. \cite[Proposition 2.5]{kubiak14}).
\end{example}

The next result shows that the uniform slice-D2P is inherited by the $\ell_\infty$-sum of spaces.

\begin{proposition}
Let $X$ be a Banach space with the uniform slice-D2P. Then, for any non-zero Banach space $Y$, the space $X\oplus_\infty Y$ has the uniform slice-D2P.
\end{proposition}

\begin{proof}
It is known that $(X\oplus_\infty Y)_\mathcal U=X_\mathcal U\oplus_\infty Y_\mathcal U$. The result follows from the fact that slice-D2P is inherited by the $\ell_\infty$-sum if one of the factors has the slice-D2P (c.f. e.g. \cite[Lemma 2.1]{lop2005}).
\end{proof}

For the $\ell_p$-sum we have the following result.

\begin{proposition} Given $\alpha>0$ and $n\in\mathbb N$, the following inequality holds
$$C_{n^2}^{\alpha}(X\oplus_p Y)\leq \left(C_n^\alpha(X)^p+C_n^\alpha(Y)^p \right)^\frac{1}{p}.$$
In particular, if $X$ and $Y$ have the uniform slice-D2P, then so does $X\oplus_p Y$.
\end{proposition}

\begin{proof}
Let $(x,y)\in S_{X\oplus_p Y}$ and let $r>0$. We can assume up to a density argument that both $x\neq 0$ and $y\neq 0$. 
Since $x\in B_X$, by definition of $C_n^\alpha(X)$, we can find $u:=\sum_{i=1}^n \lambda_i \frac{u_i+v_i}{2}$ with $\left\Vert \frac{x}{\Vert x\Vert}-u\right\Vert<d\left(\frac{x}{\Vert x\Vert},S_n^\alpha(X)\right)+r \leq C_n^\alpha(X)+r$, where $u_i,v_i\in B_X$ satisfy $\Vert u_i-v_i\Vert\geq  \alpha$ for every $1\leq i\leq n$ and $\lambda_1,\ldots, \lambda_n\in [0,1]$ are such that $\sum_{i=1}^n \lambda_i=1$.

Similarly, since $y\in B_Y$ we can find $v:=\sum_{i=1}^n \mu_i \frac{a_i+b_i}{2}$ with $\left\Vert \frac{y}{\Vert y\Vert}-v\right\Vert<C_n^\alpha(Y)+r$, where $a_i,b_i\in B_X$ satisfy $\Vert a_i-b_i\Vert\geq  \alpha$ for every $1\leq i\leq n$ and $\mu_1,\ldots, \mu_n\in [0,1]$ are such that $\sum_{i=1}^n \mu_i=1$.

Now $(\tilde u,\tilde v)=\sum_{i=1}^n\sum_{j=1}^n \lambda_i \mu_j \frac{(\Vert x\Vert u_i, \Vert y\Vert a_j)+(\Vert x\Vert v_i, \Vert y\Vert b_j)}{2}\in S_{n^2}^\alpha(X\oplus_p Y)$. Indeed, given $i,j\in\{1,\ldots, n\}$ we have
$$\Vert (\Vert x\Vert u_i,\Vert y\Vert a_j)\Vert^p=\Vert x\Vert^p\Vert u_i\Vert^p+\Vert y\Vert^p\Vert a_j\Vert^p\leq \Vert x\Vert^p+\Vert y\Vert^p=\Vert (x,y)\Vert^p=1.$$
In a similar way we obtain that $(\Vert x\Vert v_i, \Vert y\Vert b_j)\in B_{X\oplus_p Y}$. On the other hand we have
\[\begin{split}
\Vert (\Vert x\Vert u_i, \Vert y\Vert a_j)-(\Vert x\Vert v_i, \Vert y\Vert b_j)\Vert^p&  =\Vert x\Vert^p\Vert u_i-v_i\Vert^p+\Vert y\Vert^p\Vert a_j-b_j\Vert^p\\
& \geq \alpha^p(\Vert x\Vert^p+\Vert y\Vert^p)=\alpha^p.\end{split}\]
Consequently $(\tilde u,\tilde v)\in S_{n^2}^\alpha(X\oplus_p Y)$.

Finally, in order to estimate $\Vert (x,y)-(\tilde u,\tilde v)\Vert$ observe that $\tilde u=\Vert x\Vert u$. Indeed
\[
\begin{split}
\tilde u=\sum_{i=1}^n \sum_{j=1}^n \lambda_i\mu_j \frac{\Vert x\Vert u_i+\Vert x\Vert v_i}{2}& =\sum_{i=1}^n \lambda_i\left( \sum_{j=1}^n \mu_j\right) \frac{\Vert x\Vert u_i+\Vert x\Vert v_i}{2}\\
& =\sum_{i=1}^n \lambda_i \Vert x\Vert \frac{u_i+v_i}{2}\\
& =\Vert x\Vert \sum_{i=1}^n \lambda_i \frac{u_i+v_i}{2}=\Vert x\Vert u.
\end{split}
\]
With a similar argument it follows that $\tilde v=\Vert y\Vert v$.

This implies
\[\begin{split}(C_n^\alpha(X)+r)^p+(C_n^\alpha(Y)+r)^p &\geq  \left\Vert \frac{x}{\Vert x\Vert}-u \right\Vert^p +\left\Vert \frac{y}{\Vert y\Vert}-v\right\Vert^p\\
 & \geq \left\Vert \left(\frac{x}{\Vert x\Vert},\frac{y}{\Vert y\Vert}\right)-\left(\frac{\tilde u}{\Vert x\Vert} ,\frac{\tilde v}{\Vert y\Vert}\right)\right\Vert^p\\
& =\frac{\Vert x-\tilde u\Vert^p}{\Vert x\Vert^p}+\frac{\Vert y-\tilde v\Vert^p}{\Vert y\Vert^p}\\
& \geq \Vert x-\tilde u\Vert^p+\Vert y-\tilde v\Vert^p=\Vert (x,y)-(\tilde u,\tilde v)\Vert^p
\end{split}\]
since $0<\Vert x\Vert^p<1$ and $0<\Vert y\Vert^p<1$. The arbitrariness of $r>0$ and $(x,y)\in B_{X\oplus_p Y}$ proves the result.
\end{proof}

Let us continue with an example coming from \cite{ivakhno06} in the context of Lipschitz function spaces. 

\begin{example}\label{exam:Lipschitzivakhno}
Let $M$ be a metric space with a distinguished point $0\in M$ and let $\Lip(M)$ be the space of Lipschitz functions $f:M\longrightarrow \mathbb R$ which vanish at $0$ endowed with the standard Lipschitz norm (see \cite{weaver} for background). 

From the results of \cite[Section 2]{ivakhno06} it follows that if either $\inf\{d(x,y): x,y\in M, x\neq y\}=0$ or if $M$ is unbounded, then $\Lip(M)$ has the uniform slice-D2P.

Indeed, in \cite[Theorems 1 and 2]{ivakhno06} it is proved that in both the above cases then $\Lip(M)$ satisfies the hypothesis of \cite[Lemma 2]{ivakhno06}. Moreover, in the proof of the above mentioned \cite[Lemma 2]{ivakhno06} it is proved that, given any $\varepsilon
>0$ and $f\in B_{\Lip(M)}$ then, for every $n\in\mathbb N$ there are Lipschitz functions $x_1,y_1,\ldots, x_n,y_n\in (1+\varepsilon
)B_{\Lip(M)}$ such that $\Vert x_k-y_k\Vert\geq 2$ and 
$$\left\Vert f-\frac{1}{n}\sum_{k=1}^n \frac{x_k+y_k}{2}\right\Vert<\frac{4}{n}.$$
If we define $\tilde x_k:=\frac{x_k}{1+\varepsilon
}$ and $\tilde y_k:=\frac{y_k}{1+\varepsilon}$ then $\Vert\tilde x_k-\tilde y_k\Vert\geq \frac{2}{1+\varepsilon}$ and 
$$\left\Vert f-\frac{1}{n}\sum_{k=1}^n \frac{\tilde x_k+\tilde y_k}{2}\right\Vert<\frac{4}{n}+\varepsilon.$$
The arbitrariness of $f\in B_{\Lip(M)}$ reveals that
$$C_n^\frac{2}{1+\varepsilon}(\Lip(M))\leq \frac{4}{n}+\varepsilon.$$
From here the uniform slice-D2P on $\Lip(M)$ follows. Indeed, given $0<\alpha<2$ and $\delta>0$, find $m\in\mathbb N$ such that $\frac{5}{n}<\delta$ holds for every $n\geq m$. Furthermore, we can find $\varepsilon>0$ small enough to guarantee $\frac{2}{1+\varepsilon}>\alpha$ (consequently $C_n^\alpha(\Lip(M))\leq C_n^{\frac{2}{1+\varepsilon}}(\Lip(M))$ holds for every $n\in\mathbb N$) and $\varepsilon<\frac{1}{n}$. Now, given $n\geq m$, we get
$$C_n^\alpha(\Lip(M))\leq C_n^{\frac{2}{1+\varepsilon}}(\Lip(M))\leq \frac{4}{n}+\varepsilon<\frac{5}{m}<\delta.$$
Summarising we have proved that, given any $0<\alpha<2$ and any $\delta>0$ there exists $m\in\mathbb N$ such that $C_n^\alpha(\Lip(M))<\delta$ holds for every $n\geq m$. Consequently, $\Lip(M)$ has the uniform slice-D2P.
\end{example}

\begin{remark}
We want to point out that, in the paper  \cite{ivakhno06}, the author considers the Banach space quotient $\justLip(M)$ resulting from considering the space of all the Lipschitz functions over $M$ when endowed with the classical seminorm
$$L(f):=\sup\limits_{x,y\in M; x\neq y}\frac{f(x)-f(y)}{d(x,y)}.$$
However, it is well known that $\justLip(M)$ and $\Lip(M)$ are isometrically isomorphic Banach spaces regardless the choice of distinguished point $0\in M$ (c.f. e.g. \cite[p. 36]{weaver}).
\end{remark}

We end the section by exhibiting another example with the uniform slice-D2P. Throuhgout the rest of the section we will consider uniform algebras over the scalar field $\mathbb K$, either $\mathbb R$ or $\mathbb C$. Let us introduce some notation used in \cite{nw01}. Recall that a \textit{uniform algebra over a compact Hausdorff topological space $K$} is a closed subalgebra $X\subseteq \mathcal C(K)$, the space of all the continuous functions $f:K\longrightarrow \mathbb K$, which separates the points of $K$ and contains the constant functions. 

Given a uniform algebra on a compact space $K$, a point $x\in K$ is said to be a \textit{strong boundary point} if, for every neighbourhood $V$ of $x$ and every $\delta>0$, there exists $f\in S_X$ such that $f(x)=1$ and $\vert f\vert<\delta$ on $K\setminus V$. The \textit{Silov boundary} of $X$, denoted by $\partial_X$ following the notation of \cite{gamelin}, is the closure of the set of all strong boundary points. A fundamental result of the theory of uniform algebras is that $X$ can be indentified with a uniform algebra on its Silov boundary (see \cite{nw01}). This fact allows us to assume, with no loss of generality, that the Silov boundary of $X$ is dense in $K$.

Now we get the following example.

\begin{example}\label{prop:uniformalgebras}
Let $X$ be an infinite-dimensional uniform algebra. Then $X$ has the uniform slice-D2P.

Observe that in the proof of \cite[Theorem 1]{nw01} the following is proved: given a strong boundary point $x_0\in K$, an open neighbourhood $V$ of $x_0$ in $K$ and $\delta>0$ there exists $g\in B_X$ and $\varphi\in S_X$ such that
\begin{enumerate}
    \item $\vert g(t)\vert<\delta$ holds for every $t\in K\setminus V$.
    \item $h:=f(1-g)$ satisfies $\Vert h\Vert\leq 1+3\delta$.
    \item $\Vert h\pm \varphi\Vert\leq 1+4\delta$.
\end{enumerate}
Let us concluye the uniform slice-D2P from the above construction.

Since $X$ is infinite-dimensional we conclude that $K$ is infinite, so we can take a sequence of pairwise disjoint open sets $\{V_n\}$ in $K$. By the density of the Silov boundary we can take a strong boundary point $t_n\in V_n$ for every $n\in\mathbb N$. 

Let $n\in\mathbb N$ and $\delta>0$. Given $1\leq k\leq n$ consider $g_k, h_k, \varphi_k$ (associated to the strong boundary $x_k$ and the open set $V_k$) as exposed above and define 
$$a_k:=\frac{h_k+\varphi_k}{1+4\delta}; b_k:=\frac{h_k-\varphi_k}{1+4\delta}.$$
It is clear (by (3)) that $a_k,b_k\in B_X$ and, moreover,
$$\Vert a_k-b_k\Vert=\frac{2\Vert \varphi_k\Vert}{1+4\delta}=\frac{2}{1+4\delta}.$$
Hence $z:=\frac{1}{n}\sum_{k=1}^n \frac{a_k+b_k}{2}=\frac{1}{n}\sum_{k=1}^n \frac{h_k}{1+4\delta}=\frac{\frac{1}{n}\sum_{k=1}^n h_k}{1+4\delta}\in S_n^{\frac{2}{1+4\delta}}(X)$.
Let us estimate $\Vert f-z\Vert$, for which we will estimate first $\Vert f-(1+4\delta)z\Vert$. Observe that
$$f-(1+4\delta)z=f-\frac{1}{n}\sum_{k=1}^n h_k=\frac{1}{n}\sum_{k=1}^n f-h_k=\frac{1}{n}\sum_{k=1}^n f-f(1-g_k)=\frac{1}{n}\sum_{k=1}^n fg_k.$$
In order to estimate $\Vert f-(1+4\delta)z\Vert$ select $t\in K$. Since $V_i\cap V_j=\emptyset$ if $i\neq j$ we get that $t\notin V_k$ for all $k\in\{1,\ldots, k\}$ except, at most, for one $k_0$. Anyway, for every $k\neq k_0$ we get $t\notin V_k$, which in turn implies $\vert g_k(t)\vert<\delta$ (by (1)). Consequently
\[
\begin{split}
\vert f(t)-(1+4\delta)z(t)\vert& =\left\vert\frac{1}{n}\sum_{k=1}^n f(t)g_k(t) \right\vert\leq \frac{1}{n}\sum_{k=1}^n \vert f(t)\vert \vert g_k(t)\vert\\
& \leq \frac{1}{n}\sum_{k=1}^n \vert g_k(t)\vert= \frac{1}{n}\left(\vert g_{k_0}(t)\vert+\sum_{k\neq k_0}\vert g_k(t)\vert \right)\\
& <\frac{1}{n}\left(1+(n-1)\delta) \right)\leq \frac{1}{n}+\delta.
\end{split}
\]
The arbitrariness of $t\in K$ implies that $\Vert f-(1+4\delta z)\Vert\leq \frac{1}{n}+\delta$, so $\Vert f-z\Vert\leq \frac{1}{n}+5\delta$. 
The arbitrariness of $f\in B_X$ implies that
$$C_n^{\frac{2}{1+4\delta}}(X)\leq \frac{1}{n}+5\delta.$$
A similar reasoning to that of the end of Example~\ref{exam:Lipschitzivakhno} concludes that $X$ has the uniform slice-D2P, as desired.
\end{example}

\section{A Daugavet space failing the uniform slice-D2P}\label{section:counterexample}

The aim of this section is to construct a Banach space $X$ with the Daugavet property satisfying that $X_\mathcal U$ fails the slice-D2P for every free ultrafilter $\mathcal U$ over $\mathbb N$. In order to do so, we will follow the construction of a Banach space $X$ with the Daugavet property satisfying that $X_\mathcal U$ fails the Daugavet property from \cite{kw04}. Our example will be a particular case of this example by a suitable choice of scalar sequence (see below). The above mentioned construction of \cite{kw04} is in turn based on a construction of \cite{boro} of a space failing the Radon-Nikodym property but where every uniformly bounded dyadic martingale converges.

In the sequel we will follow word-by-word the construction of \cite[Section 2]{kw04}. We denote by $L_1:=L_1(\Omega,\Sigma,\mu)$ over a separable non-atomic measure space, and we will denote by $\Vert\cdot\Vert$ the canonical norm on $L_1$ throughout the section. We also consider the topology of convergence in measure, which is the one generated by the metric
$$d_m(f,g):=\inf\left\{\varepsilon>0: \mu\{t: \vert f(t)-g(t)\vert\geq \varepsilon \}\leq\varepsilon \right\}.$$

Observe that, given $f,g\in L_1$ it is immediate that $d_m(f,g)=d_m(f-g,0)$. Consequently,
\[\begin{split}
d_m(f+g,0)\leq d_m(f+g,g)+d_m(g,0)& =d_m(f+g-g,0)+d_m(g,0)\\
& =d_m(f,0)+d_m(g,0)\end{split}\]
and, inductively, $d\left(\sum_{i=1}^n f_i,0 \right)\leq \sum_{i=1}^n d_m(f_i,0)$ holds for every $f_1,\ldots, f_n\in L_1$. It is also easy to prove that given $f\in L_1$ and given $\lambda\in [0,1]$ it follows that $d_m(\lambda f,0)\leq d_m(f,0)$.
 
The following result, based on an argument of disjointness of supports of functions in $L_1$, will be used in the future. For a complete proof we refer to \cite[Lemma 2.1]{kw04}.

\begin{lemma}\label{lemma:sectexl1orthog}
Let $H$ be a uniformly integrable subset of $L_1$ and $\varepsilon>0$. Then there exists $\delta>0$ such that, if $g\in H$ and $f\in L_1$ satisfies $d_m(f,0)<\delta$ then
$$\Vert f+g\Vert\geq\Vert f\Vert+\Vert g\Vert-\varepsilon.$$
\end{lemma}

The following lemma, whose proof can be found in \cite[Lemma 5.26]{beli}, is essential in the future construction.

\begin{lemma}\label{lemma:belicounter}
Let $0<\varepsilon<1$. Then there exists a function $f\in L_1([0,1])$ such that
\begin{enumerate}
    \item $f\geq 0$, $\Vert f\Vert=1$ and $\Vert f-{\bf{1}}\Vert\geq 2-\varepsilon$.
    \item Let $\{f_j\}$ be a sequence of independent random variables with the same distribution as $f$. If $g\in \overline{\spann}\{f_j\}$ with $\Vert g\Vert\leq 1$ then there exists a constant function $c$ with $d_m(g,c)\leq \varepsilon$.
    \item  $\left\Vert \frac{1}{n}\sum_{j=1}^n f_j-{\bf{1}}\right\Vert\rightarrow 0$ as $n\rightarrow\infty$.
\end{enumerate}
\end{lemma}

In the lemma and in the construction below we consider $(\Omega,\Sigma,\mu)$ as the product of countably many copies of the measure space $[0,1]$. 

We say that a subspace $X$ of $L_1$ \textit{depends on finitely many coordinates} if all $f\in X$ are functions depending on a finite common set of coordinates.

Now we consider the following lemma, which appears in \cite{kw04} (see \cite[Lemma 2.4]{kw04} for a proof).

\begin{lemma}\label{lemma:counterprevioconstruc}
Let $G$ be a finite dimensional subspace of $L_1$ that depends on finitely many coordinates. Let $\{u_k\}_{k=1}^m\subseteq S_G$ and $\varepsilon>0$. Then there exists a finite dimensional subspace $F$ of $L_1$ containing $G$ and depending on finitely many coordinates and there exist $n\in\mathbb N$ and functions $\{v_{k,j}\}_{k\leq m, j\leq n}$ such that:
\begin{enumerate}
    \item $\Vert u+v_{k,j}\Vert\geq 2-\varepsilon$ holds for every $u\in S_G$ and all $k\leq m$ and $j\leq n$,
    \item $\left\Vert u_k-\frac{1}{n}\sum_{j=1}^n v_{k,j} \right\Vert\leq \varepsilon$ for every $k$,
    \item For every $\varphi\in B_F$ there exists $\psi\in B_G$ with $d_m(\varphi,\psi)\leq \varepsilon$.
\end{enumerate}
\end{lemma}

Now we will make the construction of the space. Fix a decreasing sequence $(\varepsilon_N)$ of positive numbers with $\sum_{j=N+1}^\infty \varepsilon_j<\varepsilon_N$ for all $N\in\mathbb N$ and select inductively finite-dimensional subspaces of $L_1$,
$$\spann{{\bf{1}}}=E_1\subset E_2\subset E_3\subset \ldots ,$$
each of them depending on finitely many coordinates, $\varepsilon_N$-nets $\{u_k^N\}_{k=1}^{m(N)}$ of $S_{E_N}$ and collections of elements $\{v_{k,j}^N\}_{k\leq m(N), j\leq n(N)}$ in such a way that the conclusion of Lemma \ref{lemma:counterprevioconstruc} holds with $\varepsilon=\varepsilon_N$, $G=E_N$, $F=E_{N+1}$, $\{u_k\}_{k=1}^m=\{u_k^N\}_{k=1}^{m(N)}$, $\{v_{k,j}\}_{k\leq m, j\leq n}=\{v_{k,j}^N\}_{k\leq m(N), j\leq n(N)}$. Denote $E:=\overline{\bigcup\limits_{N=1}^\infty E_N}$.

The above space $E$ satisfies the following properties, obtained from \cite[Theorem 2.5]{kw04}.

\begin{theorem}\label{theo:counterpropiexamenkawe}
The space $E$ constructed as above satisfies the following properties:
\begin{enumerate}
    \item $E$ has the Daugavet property,
    \item For every $f\in B_E$ and every $N\in\mathbb N$ there exists $g\in B_{E_N}$ satisfying that $d_m(f,g)<\varepsilon_N$,
    \item $E$ has the Schur property.
\end{enumerate}
\end{theorem}

In \cite[Theorem 3.3]{kw04} the authors make use of the above space in order to construct a Banach space $X$ with the Daugavet property such that $X_\mathcal U$ fails the Daugavet property for every free ultrafilter $\mathcal U$ over $\mathbb N$. In the following, we will make use of many of their ideas in order to prove the following theorem.

\begin{theorem}\label{theo:examfallasliceuniforme}
Let $n\in\mathbb N$ and $\eta>0$. There exists a Banach space $X$ with the Daugavet property such that 
$$C_n^{2\eta}(X)\geq \frac{\eta}{8}.$$
\end{theorem}

\begin{proof}
Select $\delta>0$ small enough so that 
$$5\delta<\frac{\eta}{2}.$$
Let $X$ be the space of Theorem \ref{theo:counterpropiexamenkawe} with $\varepsilon_1>0$ small enough to satisfy that given any constant function $g\in [-2,2]$ (i.e. $g\in E_1$) and $f\in L_1$, the condition $d_m(f,0)<2n\varepsilon_1$ implies
\begin{equation}\label{eq:countereleccionparametros}
\Vert f+g\Vert\geq \Vert f\Vert+\Vert g\Vert-\delta.    
\end{equation}
Our aim is to prove that
\begin{equation}\label{eq:targetlemaelepara}
d\left({\bf{1}}, S_n^{2\eta}(X)\right)\geq \frac{\eta}{8}.
\end{equation}
In order to do so take $z\in S_n^{2\eta}(X)$. Then $z=\sum_{k=1}^n \lambda_k z_k$ with $z_k\in S^{2\eta}(X)$ and $\lambda_1,\ldots, \lambda_n\in [0,1]$ with $\sum_{k=1}^n \lambda_k=1$. Moreover, since $z_k\in S^{2\eta}(X)$ it follows that $z_k=\frac{u_k+v_k}{2}$ with $u_k,v_k\in B_X$ satisfying that $\Vert u_k-v_k\Vert\geq 2\eta$ holds for every $1\leq k\leq n$. Now given $k$, the triangle inequality implies
$$2\eta\leq \Vert u_k-{\bf{1}}+{\bf{1}}-v_k\Vert\leq \Vert {\bf{1}}-u_k\Vert+\Vert {\bf{1}}-v_k\Vert.$$
The above inequality implies that either $\Vert {\bf{1}}-u_k\Vert\geq \eta$ or $\Vert {\bf{1}}-v_k\Vert\geq \eta$. Assume, up to a relabeling, that $\Vert {\bf{1}}-u_k\Vert\geq\eta$ holds for every $1\leq k\leq n$.

Now, given $1\leq k\leq n$ apply (b) of Theorem \ref{theo:counterpropiexamenkawe} (applied to $f=u_k$ and $v_k$ respectively and $N=1$) to find constant functions $\alpha_k, \beta_k\in [-1,1]$ satisfying $d_m(u_k,\alpha_k)<\varepsilon_1$ and $d_m(v_k,\beta_k)<\varepsilon_1$. 

Now, given $1\leq k\leq n$, we have
$$1\geq \Vert u_k\Vert=\Vert \alpha_k+(u_k-\alpha_k)\Vert>\vert \alpha_k\vert+\Vert u_k-\alpha_k\Vert-\delta$$
since $\alpha_k$ is a constant function and $d_m(u_k-\alpha_k,0)=d_m(u_k,\alpha_k)<\varepsilon_1<2n\varepsilon_1$, so the inequality \eqref{eq:countereleccionparametros} holds. Now
\[\begin{split}1\geq \vert\alpha_k\vert+\Vert u_k-{\bf{1}}+{\bf{1}}-\alpha_k\Vert-\delta& \geq \vert\alpha_k\vert+\Vert {\bf{1}}-u_k\Vert-\vert 1-\alpha_k\vert-\delta\\
& =\vert\alpha_k\vert+\Vert {\bf{1}}-u_k\Vert-(1-\alpha_k)-\delta,
\end{split}\]
where the last inequality follows since $\alpha_k\leq 1$. Taking into account that $\Vert {\bf{1}}-u_k\Vert\geq \eta$ the above inequality implies
$$1\geq \vert \alpha_k\vert+\eta-(1-\alpha_k)-\delta=\vert \alpha_k\vert+\alpha_k+\eta-1-\delta\geq 2\alpha_k-1+\eta-\delta.$$
Consequently
$$2\alpha_k\leq 2-\eta+\delta\Rightarrow \alpha_k\leq \frac{2-\eta}{2}+\frac{\delta}{2}.$$
Since $\beta_k\in [-1,1]$ holds for every $k$ we get
\begin{equation}\label{eq:sumalphabeta}\sum_{k=1}^n \lambda_k \frac{\alpha_k+\beta_k}{2}\leq \frac{\frac{2-\eta}{2}+\frac{\delta}{2}+1}{2}=\frac{4-\eta+\delta}{4}.
\end{equation}
Now
\[\begin{split}
d_m\left(z-\sum_{k=1}^n \lambda_k \frac{\alpha_k+\beta_k}{2},0\right)& =d_m\left(\sum_{k=1}^n \frac{\lambda_k}{2}(u_k-\alpha_k+v_k-\beta_k)\right)\\
& \leq \sum_{k=1}^n d_m(u_k-\alpha_k,0)+d_m(v_k-\beta_k,0)<2n\varepsilon_1.
\end{split}\]
If we apply \eqref{eq:countereleccionparametros} to the constant function ${\bf{1}}-\sum_{k=1}^n \lambda_k \frac{\alpha_k+\beta_k}{2}$ and the function $z-\sum_{k=1}^n \frac{\alpha_k+\beta_k}{2}$, which is $2n\varepsilon_1$ close to $0$ with respect to the distance $d_m$, we obtain
\[
\begin{split}
\left\Vert {\bf{1}}-z \right\Vert& =\left\Vert \left({\bf{1}}-\sum_{k=1}^n \lambda_k \frac{\alpha_k+\beta_k}{2}\right)-\left(z-\sum_{k=1}^n \lambda_k \frac{\alpha_k+\beta_k}{2}\right) \right\Vert\\
& \geq \left\Vert {\bf{1}}-\sum_{k=1}^n \lambda_k \frac{\alpha_k+\beta_k}{2} \right\Vert +\left\Vert z-\sum_{k=1}^n \lambda_k\frac{\alpha_k+\beta_k}{2} \right\Vert-\delta\\
& \geq \left\Vert {\bf{1}}-\sum_{k=1}^n \lambda_k \frac{\alpha_k+\beta_k}{2} \right\Vert-\delta\\
& \geq 1-\sum_{k=1}^n \lambda_k\frac{\alpha_k+\beta_k}{2}-\delta\mathop{\geq}\limits^{\mbox{\eqref{eq:sumalphabeta}
}} 1-\frac{4-\eta+\delta}{4}-\delta=\frac{\eta-5\delta}{4}>\frac{\eta}{8}.
\end{split}
\]
Now the result follows by the arbitrariness of $z\in S_n^{2\eta}(X)$.\end{proof}

Let $\eta>0$ and, for every $n\in\mathbb N$, consider $X_n$ as the Banach space claimed in Theorem \ref{theo:examfallasliceuniforme}, and consider $X=\left( \oplus_{n=1}^\infty X_n\right)_1$. $X$ has the Daugavet property as it is an $\ell_1$-sum of Banach spaces with the Daugavet property \cite[Theorem 1]{woj92}. Let $r>0$ small enough to guarantee $2r<\eta$ and $\frac{r^2}{4}+r<\frac{\eta}{8}$. We claim that, given $n\in\mathbb N$, we get that
$$d\left((0,0,0,\ldots, \underbrace{{\bf{1}}}
\limits_{n}, 0,0,\ldots), S_n^{3\eta}(X)\right)\geq \frac{r^2 }{4}.$$

In order to prove it write $x:=(0,0,0,\ldots, \underbrace{{\bf{1}}}
\limits_{n}, 0,0,\ldots)$ and assume by contradiction that there is $z\in S_n^{3\eta}(X)$ such that $\Vert x-z\Vert<\left(\frac{r}{2}\right)^2$. Consequently
$$\Vert 1-z(n)\Vert=\Vert x(n)-z(n)\Vert\leq \sum_{k=1}
^\infty \Vert x(k)-z(k)\Vert=\Vert x-z\Vert\leq \left(\frac{r}{2}\right)^2.$$
If we write $z=\sum_{i=1}^n \lambda_i z_i$ for $0\leq \lambda_i\leq 1$ with $\sum_{i=1}^n \lambda_i=1$ and $z_i\in S^{3\eta}(X)$, we obtain from the above inequality that $\left\Vert \sum_{i=1}^n \lambda_i z_i(n)\right\Vert>1-\left(\frac{r}{2}\right)^2$. Set
$$G:=\left\{i\in\{1,\ldots, n\}: \Vert z_i(n)\Vert>1-\frac{r}{2}\right\}$$
We claim that $\sum_{i\notin G}\lambda_i<\frac{r}{2}$. Indeed,
\[
\begin{split}
1-\left(\frac{r}{2}\right)^2<\sum_{i=1}^n \lambda_i \Vert z_i(n)\Vert& =\sum_{i\in G}\lambda_i \Vert z_i(n)\Vert+\sum_{i\notin G}\lambda_i \Vert z_i(n)\Vert\\
& \leq \sum_{i\in G}\lambda_i+\sum_{i\notin G}\lambda_i\left(1-\frac{r}{2}\right)\\
& =1-\frac{r}{2}\sum_{i\notin G}\lambda_i,
\end{split}
\]
from where $\sum_{i\notin G}\lambda_i<\frac{r}{2}$ follows.

On the other hand, since $z_i\in S^{3\eta}(X)$ then for $1\leq i\leq n$ there are $u_i, v_i\in B_X$ with $z_i=\frac{u_i+v_i}{2}$ and $\Vert u_i-v_i\Vert>3\eta$. Given $i\in G$ we have $\Vert z_i(n)\Vert>1-\frac{r}{2}$, from where 
$$1-\frac{r}{2}<\frac{\Vert u_i(n)+v_i(n)\Vert}{2}\leq \frac{\Vert u_i(n)\Vert+\Vert v_i(n)\Vert}{2},$$
and an easy convexity argument implies $\Vert u_i(n)\Vert>1-r$ and $\Vert v_i(n)\Vert>1-r$. Consequently, given $i\in G$ we have
$$1-r<\Vert u_i(n)\Vert\leq \Vert u_i(n)\Vert+ \sum_{k\neq n}\Vert u_i(k)\Vert\leq \Vert u_i\Vert\leq 1,$$
from where $\sum_{k\neq n}\Vert u_i(k)\Vert<r$. Similarly $\sum_{k\neq n}\Vert v_i(k)\Vert<r$. Since $\Vert u_i-v_i\Vert>3\eta$ and $2r<\eta$ we obtain
$$3\eta<\Vert u_i(n)-v_i(n)\Vert+\sum_{k\neq n}\Vert u_i(k)\Vert+\Vert v_i(k)\Vert\leq \Vert u_i(n)-v_i(n)\Vert+2r,$$
so $\Vert u_i(n)-v_i(n)\Vert >3\eta-2r>2\eta$ holds for every $i\in G$. Set $\lambda:=1-\sum_{i\in G}\lambda_i$ and set $z':=\sum_{i\in G}\lambda_i z_i+\lambda z$ where $z=z_{i_0}$ for any $i_0\in G$. We clearly get that $z'(n)=\sum_{i\in G}\lambda_i \frac{u_i(n)+v_i(n)}{2}+\lambda \frac{u_{i_0}(n)+v_{i_0}(n)}{2}$ where $\Vert u_i(n)-v_i(n)\Vert>2\eta$ and $\Vert u_{i_0}(n)-v_{i_0}(n)\Vert>2\eta$. This means $z'(n)\in S_n^{2\eta}(X_n)$. By \eqref{eq:targetlemaelepara} we obtain
$$\Vert {\bf{1}}-z'(n)\Vert\geq \frac{\eta}{8}.$$ Consequently
\[
\begin{split}
\frac{\eta}{8}\leq \Vert x(n)-z'(n)\Vert& \leq \Vert x-z'\Vert\leq \Vert x-z\Vert+\Vert z'-z\Vert\\
& \leq \frac{r^2}{4}+\sum_{i\notin G}\lambda_i\left \Vert z_i-\frac{u_{i_0}+v_{i_0}}{2}\right\Vert< \frac{r^2}{4}+r<\frac{\eta}{8},
\end{split}\]
a contradiction. 

This proves that for every $n\in\mathbb N$ it follows
$$C_n^{3\eta}(X)\geq \frac{r^2}{2}.$$
According to Theorem \ref{theo:necebigdiamultrapower} we have proved the following result.

\begin{theorem}\label{theo:ejemfinalnounislice}
For every $\eta>0$ there exists a Banach space $X$ with the Daugavet property such that, for every free ultrafilter $\mathcal U$ over $\mathbb N$, the space $(X)_\mathcal U$ has a slice of diameter smaller than or equal to $\eta$.
\end{theorem}
 
\section*{Acknowledgements}  

This work was supported by MCIN/AEI/10.13039/501100011033: grant PID2021-122126NB-C31, Junta de Andaluc\'ia: grant FQM-0185, by Fundaci\'on S\'eneca: ACyT Regi\'on de Murcia: grant 21955/PI/22 and by Generalitat Valenciana: grant CIGE/2022/97.

\end{document}